%% file: main.tex
\documentclass[12pt]{article}
\usepackage[latin1]{inputenc}
\usepackage{amsmath}
\usepackage{amsfonts}
\usepackage{amssymb}
\usepackage{booktabs}
\usepackage{multirow}
\usepackage[american]{babel}
\usepackage{graphicx}
\usepackage{amsthm}
\usepackage{color}
\usepackage{xcolor}
\usepackage[normalem]{ulem}
\usepackage[text={5in,7.125in},centering]{geometry}
\usepackage{hyperref}
\usepackage{caption}
\usepackage{float}
\usepackage{tikz}
\usetikzlibrary{shapes.geometric}

\begin{document}

\title{Manhattan and Chebyshev flows}{}

\newtheorem{theorem}{Theorem}
\newtheorem{proposition}[theorem]{Proposition}
\newtheorem{lemma}[theorem]{Lemma}
\newtheorem{definition}[theorem]{Definition}
\newtheorem{example}[theorem]{Example}
\newtheorem{corollary}[theorem]{Corollary}	
\newtheorem{conjecture}[theorem]{Conjecture}
\newtheorem{remark}{Remark}
\newtheorem{problem}{Problem}
\newtheorem{claim}{Claim}

\author{Luk\'a\v s G\'aborik\footnote{Comenius University, Bratislava, Slovakia, \texttt{gaborik5@uniba.sk}} , Sascha Kurz\footnote{Universit\"{a}t Bayreuth, Bayreuth, Germany, \texttt{Sascha.Kurz@uni-bayreuth.de}} , Giuseppe Mazzuoccolo\footnote{Universit\`{a} di Modena e Reggio Emilia, Italy, \texttt{giuseppe.mazzuoccolo@unimore.it}} ,\\Jozef Rajn\'{i}k\footnote{Comenius University, Bratislava, Slovakia, \texttt{jozef.rajnik@fmph.uniba.sk}}, Florian Rieg\footnote{Goethe University Frankfurt, Germany, \texttt{rieg@math.uni-frankfurt.de}}}

\vspace*{-2cm}
{\let\newpage\relax\maketitle}
\maketitle

\begin{abstract}
We investigate multidimensional nowhere-zero flows of bridgeless graphs. 
By extending the established use of the Euclidean norm, this paper considers the Manhattan and Chebyshev norms, leading to the definition of the flow numbers $\Phi_d^1(G)$ and $\Phi_d^\infty(G)$, respectively.
These flow numbers are always rational and in two dimensions, they distinguish between cubic graphs that are 3-edge-colourable and those that are not. We also prove that, for any bridgeless graph $G$, the two values $\Phi^1_2(G)$ and $\Phi^\infty_2(G)$ are the same. We give new upper and lower bounds and structural results, and we find connections with cycle covers. Finally, we introduce the idea of $t$-flow-pairs, which comes from a method used in Seymour's proof of the 6-flow theorem, and we propose new conjectures that could be stronger than Tutte's famous 5-flow conjecture.

\end{abstract}

\section{Introduction}

A $\Gamma$-flow, for an (additive) Abelian group $\Gamma$, consists of an orientation of the edges of a graph $G$ and a mapping $\varphi\colon E(G) \rightarrow \Gamma$ such that for each vertex $v$ of $G$, the sum of the incoming values equals the sum of the outgoing values.
When the orientation is clear from the context, we often identify the flow with the mapping $\varphi$.
The set of allowed flow values is often restricted to avoid trivial cases.
If the flow value of each edge is non-zero, such a flow is called a \emph{nowhere-zero $\Gamma$-flow}.
Only a bridgeless graph can admit a nowhere-zero $\Gamma$-flow.
More generally, we can restrict the flow values allowed to some subset of $\Gamma$. For instance, a \emph{nowhere-zero $k$-flow} is a $\mathbb{Z}$-flow with the allowed set $\{\pm 1, \pm 2, \dots, \pm (k - 1)\}$, a \emph{circular $r$-flow} is an $\mathbb{R}$-flow with values from $[- r + 1, -1] \cup [1, r - 1]$, and a \emph{unit-vector} flow is an $\mathbb{R}^d$-flow using only unit vectors, for some integer $d \ge 1$. 

Recently, Mattiolo et al. \cite{complex_flows} introduced a \emph{$d$-dimensional nowhere-zero $r$-flow} (an $(r, d)$-NZF for short) as an $\mathbb{R}^d$-flow using vectors whose Euclidean norm lies in the interval $[1, r - 1]$. This concept generalises circular flows ($d=1$), complex flows ($d=2$) and unit-vector flows ($r=2$);  see \cite{7d_flows}. The $d$-dimensional flow number $\Phi_d(G)$ of a bridgeless graph $G$ is the infimum of all $r$ for which $G$ admits a $(r, d)$-NZF. 

The following conjecture by Tutte is the central open problem in this area.

\begin{conjecture} \label{conj:Tutte} \emph{\cite{tutte}}
    Let $G$ be a bridgeless graph. Then, $\Phi_1(G) \leq 5$.\label{conj:1d_flows}
\end{conjecture}

An analogous conjecture in higher dimensions was proposed by Jain. 

\begin{conjecture} \label{conj:Jain} \emph{\cite{3d_flows}}
    Let $G$ be a bridgeless graph. Then, $\Phi_d(G) = 2$ for any $d\geq 3$.\label{conj:3d_flows}
\end{conjecture}

Conjectures \ref{conj:1d_flows} and  \ref{conj:3d_flows} leave the case of $2$-dimensional flows as a natural subject of investigation. Mattiolo et al.\ \cite{complex_flows} proved that for every bridgeless graph~$G$, we have $\Phi_2(G) \le 1 + \sqrt{5}$, and left as a problem whether this upper bound can be improved to $\Phi_2(G) \le 1 + \left(1 + \sqrt{5}\right)/2$, i.e., the square of the golden ratio.

Two major differences between the 1-dimensional and 2-dimensional cases were already identified in \cite{complex_flows}.
First of all, despite the use of real numbers, the 1-dimensional flow number is always a rational number, see~\cite{goddyn_fractional_nzf}. Furthermore, in the $1$-dimensional case, the flow number directly distinguishes between $3$-edge-colourable cubic graphs, which have flow number at most $4$, and non-$3$-edge-colourable ones, which necessarily have the flow number strictly greater than $4$.
The $2$-dimensional flows lack both of these properties. For instance, the complete graph $K_4$ is $3$-edge-colourable, yet $\Phi_2(K_4) = 1 + \sqrt{2}$. On the other hand, there exist non-$3$-edge-colourable cubic graphs (for example the Flower snark $J_5$, see~\cite[p. 9]{complex_flows}) with smaller $2$-dimensional flow numbers.

A natural source of irrationality in the $2$-dimensional flow number arises from the use of the Euclidean norm in its definition. So far, multidimensional flows have only been studied in this setting. This is an initial motivation for the exploration of alternative geometric frameworks. In this paper, we initiate such a study by examining multidimensional flows defined using the \emph{Manhattan} and \emph{Chebyshev} norms. The other main motivation for this approach is the idea that flow values computed in different norms may yield lower bounds for the Euclidean case, offering new insight into its structure. We further elaborate on this idea in Section \ref{sec:openproblems}.

Recall that $\|\boldsymbol{x}\|_1 = \sum_{i = 1}^d |x_i|$ and $\|\boldsymbol{x}\|_\infty = \max_{i = 1}^d |x_i|$ are known as Manhattan and Chebyshev norm of $\boldsymbol{x} = (x_1, \dots, x_d)$, respectively.
Let $G=(V, E)$ be a bridgeless graph. Let $d$ be a positive integer and $r\geq 2$ be a real constant. A \emph{$d$-dimensional Manhattan nowhere-zero $r$-flow} (or $(r, d)$-MNZF for short) is an $\mathbb{R}^d$-flow $\varphi$ such that $1\leq\|\varphi(e)\|_1\leq r-1$ for each edge $e\in E$. Analogously, a \emph{$d$-dimensional Chebyshev nowhere-zero $r$-flow} (or $(r, d)$-ChNZF for short) satisfies the condition $1\leq\|\varphi(e)\|_\infty\leq r-1$ for each edge $e\in E$.
A \emph{$d$-dimensional Manhattan flow number} of $G$ and a \emph{$d$-dimensional Chebyshev flow number} of $G$ are
\begin{align*}
    \Phi_d^1(G) &:= \inf\{r\mid\exists(r, d)\text{-MNZF on }G\}\text{ and }\\ \Phi_d^\infty(G) &:= \inf\{r\mid\exists(r, d)\text{-ChNZF on }G\},
\end{align*}
respectively. As in the Euclidean case, the infimum can be replaced by a minimum \cite[p. 2]{complex_flows}.

In this paper, we show that these flows provide a better generalisation of circular flows in the two-dimensional case: indeed, the flow number is always rational and strictly smaller for $3$-edge-colourable graphs than for non-colourable ones. The former condition holds for any dimension.

Section~\ref{sec:rationality} is devoted to the proof of rationality and providing some normal forms of the flows. 
Moreover, we prove that $\Phi^1_2(G) = \Phi^{\infty}_2(G)$ for every bridgeless graph $G$ (see Corollary~\ref{cor:manhequalcheb}).

In Section \ref{sec:upperbounds}, we study upper bounds for these parameters depending on the dimension and explore how their values relate to the existence of certain cycle covers in bridgeless graphs.

Section~\ref{sec:2dim} focuses on the two-dimensional case. Beyond characterising $3$-edge-colourable cubic graphs, we establish a lower bound of $\Phi^\infty_2(G) \ge 2 + 1/\lfloor (n - 2)/4 \rfloor$ for any bridgeless graph $G$ of order $n$ that is not $3$-edge-colourable. This is an analogy to the bound provided in~\cite[p. 14]{cycle_rank} for the circular flow number.

Section \ref{sec:circular_decompositions} introduces the notion of $t$-flow-pairs as a generalisation of the idea underlying Seymour's proof of the $6$-flow theorem. We also propose sufficient conditions that may lead to improved upper bounds for $\Phi^1_2(G)$ and $\Phi^1_1(G)$. This also brings an interesting approach to Tutte's $5$-flow conjecture.

The final section of the paper is devoted to open problems and directions for future research.

\section{Rationality of flow numbers}\label{sec:rationality}

In this section, we prove that the two parameters $\Phi_d^\infty(G)$ and $\Phi_d^1(G)$ are always rational numbers, in contrast to what occurs in the Euclidean case. Moreover, we show that the flow values can be chosen as integer multiples of some suitable rational numbers.

\begin{theorem}
    \label{th:rationality}
    For each bridgeless graph $G$ and each integer $d \ge 1$, the values $\Phi_d^\infty(G)$ and $\Phi_d^1(G)$ are rational.
\end{theorem}

\begin{proof}
We first prove the rationality of $\Phi_d^1(G)$. The idea is to divide the space of flows into finitely many regions such that in each region, the inequalities $1\leq \|\varphi(e)\|_1\leq r-1$ are linear (in $\varphi$) for all edges $e$.
In order to emphasize that flows should be thought of as vectors in this context, we identify functions $\varphi\colon E\to\mathbb{R}^d$ with elements $x\in \mathbb{R}^{d|E|}$ and use the notation $x_{e,i}$ to refer to the $i$-th coordinate of $\varphi(e)$.
Now for each vector $\sigma\in \{-1,1\}^{d|E|}$,  let $M_{\sigma}$ be the set of all elements $(x,r)\in \mathbb{R}^{d|E|}\times \mathbb{R}$ satisfying the following conditions:
\begin{align*}
\sum_{e\in\partial^+(v)}x_{e,i}=\sum_{e\in\partial^-(v)}x_{e,i} &\qquad \text{ for all }v\in V, i\in\{1,\dots,d\}\\
\sigma_{e,i}\cdot x_{e,i}\geq 0 &\qquad \text{ for all }e\in E,i\in\{1,\dots,d\}\\
1\leq \sum_{i=1}^d \sigma_{e,i}\cdot x_{e,i}\leq r-1 &\qquad \text{ for all }e\in E
\end{align*}
Here, $\sigma$ represents the vector of signs of $x$. Note that since $\sigma$ is constant, these (in-)equalities are all linear in $x$ and $r$ and the coefficients are integers. Next, define
 \[M= \bigcup_{\sigma\in \{-1,1\}^{d|E|}}M_{\sigma}\]
and let $\pi\colon M\to \mathbb{R}$ be the projection onto the $r$-coordinate.
For a fixed value of $r$, the set $\pi^{-1}(r)$ is clearly in bijection to the set of all $(r,d)$-MNZF on $G$ via the identification above.
Thus, it holds that \[\Phi_d^1(G)=\min\,\pi(M)=\min_{\sigma\in \{-1,1\}^{d|E|}}\min\,\pi(M_{\sigma}).\]
Now the key is that $\min\,\pi(M_{\sigma})$ can be regarded as the optimum of a linear program with integral coefficients that is bounded due to Seymour's $6$-flow theorem and is thus rational. But this means that $\Phi_d^1(G)$ is the minimum of finitely many rational values and thus rational as well. 

The argument for the rationality of $\Phi_d^{\infty}(G)$ is very similar. In order to linearize the Chebyshev norm, one needs to subdivide the space further such that in each region, a fixed coordinate is maximal for all points in the region. Since these are still finitely many regions along with integral constraints for each region, the same conclusion as above holds.
\end{proof}

\begin{proposition}
	Consider positive integers $d$, $p$, $q$ and a bridgeless graph $G$ with $\Phi_d^\infty(G)=p/q$. Then there exists a $(p/q, d)$-ChNZF of $G$ such that for each edge, its flow value $\varphi(e)$ can be written as $(k_1/q, \dots, k_d/q)$, where $k_1, \dots, k_d$ are integers. 
    \label{prop:chebyshev_normal_form}
\end{proposition}

\begin{proof}
For the sake of contradiction, assume that there exists a $(p/q, d)$-ChNZF on $G$ in which the total number of coordinates not of the form $k/q$, summed over all edges, is minimal but non-zero.

Now, consider an edge $e$ whose $i^{\text{th}}$ coordinate cannot be written in the required form. Observe that for any vertex, there cannot be exactly one incident edge whose $i^{\text{th}}$ coordinate is not a multiple of $1/q$; otherwise, the flow conservation constraint at that vertex would fail to yield zero in the $i^{\text{th}}$ coordinate.

It follows that we can construct a circuit $C$ containing $e$ such that the $i^{\text{th}}$ coordinate of no edge in $C$ is a multiple of $1/q$. Let $\alpha > 0$ denote the minimum distance between the $i^{\text{th}}$ coordinates of edges in $C$ and the values $\pm 1$, $\pm(p/q - 1)$.

By adding or subtracting $\alpha$ from the $i^{\text{th}}$ coordinate of the flow values along the entire circuit $C$, we reduce the total number of conflicting coordinates, contradicting the assumption of minimality. This completes the proof.

\end{proof}

It is well known that the Manhattan and Chebyshev norms are dual to each other. This suggests a potential relationship between Manhattan flows and Chebyshev flows. In the two-dimensional case, this relationship is particularly strong.

\begin{proposition}
    A bridgeless graph has an $(r, 2)$-MNZF if and only if that graph has an $(r, 2)$-ChNZF. \label{prop:manhattan_chebyshev_equivalence}
\end{proposition}

\begin{proof}
    Consider annuli between circles of radii $1$ and $r-1$, assuming Manhattan and Chebyshev norm, separately (see Figure~\ref{fig:manhattan_chebyshev_equivalence}). These two annuli are similar with the similarity coefficient equal to $\sqrt2$ (as a ratio of the diagonal length to the side length).
    \begin{figure}[ht]
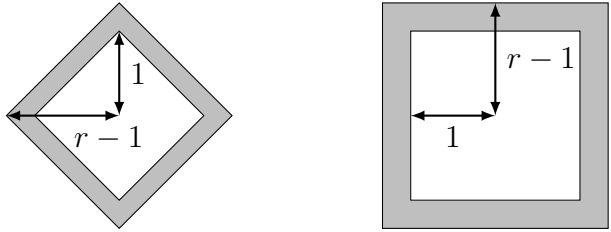

    \input images/manhattan_chebyshev_equivalence.tex
    \caption{The annulus in the Manhattan norm (left) with respect to the annulus in the Chebyshev norm (right)}
    \label{fig:manhattan_chebyshev_equivalence}
    \end{figure}
Therefore, we can transform the flow vectors in the Manhattan flow into corresponding flow vectors in the Chebyshev flow, and vice versa, using a spiral similarity. More precisely, the equations for the linear transformation under consideration are as follows:
    \[
\begin{bmatrix}
x' \\
y'
\end{bmatrix}
=
\begin{bmatrix}
1 & -1 \\
1 & 1
\end{bmatrix}
\begin{bmatrix}
x \\
y
\end{bmatrix}
=
\begin{bmatrix}
x - y \\
x + y
\end{bmatrix}
\]
    
Since the transformation is linear, it is clear that the flow conservation constraints remain satisfied.
\end{proof}

\begin{corollary}\label{cor:manhequalcheb}
    For any bridgeless graph $G$, $\Phi^1_2(G)=\Phi^\infty_2(G)$.
\end{corollary}

Although Proposition \ref{prop:manhattan_chebyshev_equivalence} establishes the equivalence of Chebyshev and Manhattan flows in two dimensions, we are not aware of any argument that extends this result to higher dimensions. The reasoning used in the two-dimensional case does not generalize, as in three dimensions, Manhattan and Chebyshev spheres are fundamentally different and cannot be transformed into one another via a linear transformation. Specifically, while the Manhattan sphere corresponds to an octahedron, the Chebyshev sphere is a cube.

Using Proposition \ref{prop:manhattan_chebyshev_equivalence}, we obtain a similar (though weaker) normal form for two-dimensional flows in the Manhattan norm.
\begin{proposition}
	Consider positive integers $p, q$ and a bridgeless graph $G$ with $\Phi_2^1(G)=p/q$. Then there exists a $(p/q, 2)$-MNZF on $G$ such that for each edge, its flow value $\varphi(e)$ can be written as $(k_1/(2q), k_2/(2q))$ for some integers $k_1, k_2$.
\end{proposition}
\begin{proof}
Let $\varphi$ be a $(p/q,2)$-MNZF on $G$. By Proposition \ref{prop:manhattan_chebyshev_equivalence}, $G$ has a $(p/q,2)$-ChNZF. Applying Proposition \ref{prop:chebyshev_normal_form}, we obtain a $(p/q,2)$-ChNZF $\varphi'$ with values of the form $(k_1/q,k_2/q)$.
Note that the linear map $\lambda$ from the unit ball in the Chebyshev norm to the unit ball in the Manhattan norm is given by

 \[
\begin{bmatrix}
x' \\
y'
\end{bmatrix}
=
\begin{bmatrix}
1/2 & -1/2 \\
1/2 & 1/2
\end{bmatrix}
\begin{bmatrix}
x \\
y
\end{bmatrix}
=
\begin{bmatrix}
\frac12(x - y) \\

\frac12(x + y)
\end{bmatrix}
\].

Thus, the composition $\lambda\circ \varphi'$ is a $(p/q,2)$-MNZF with values of the form $(k_1/(2q),k_2/(2q))$.
\end{proof}

As already remarked, such a result provides a normal form for the Manhattan flows only in two dimensions. Moreover, in contrast to the analogous result for the Chebyshev norm, it is generally not sufficient to take only multiples of $1/q$ in the previous proposition. For instance, we have $\Phi_2^1(K_4)=2$, but there does not exist a $(2,2)$-MNZF on $K_4$ with only integral entries. Such a flow would be restricted to the vectors $(\pm 1,0) $ and $(0,\pm 1)$, but no three of these sum to zero.

\section{General upper bounds by dimension}\label{sec:upperbounds}

Seymour proved that each bridgeless graph admits a nowhere-zero $6$-flow. More precisely, he obtained his main result by proving the existence of a pair of flows with some nice properties.

\begin{lemma} \emph{\cite[p. 132]{seymour}}
    For any bridgeless graph, there exist a $2$-flow $\varphi_2$ and a $3$-flow $\varphi_3$ such that for each edge $e$, at least one value from $\varphi_2(e), \varphi_3(e)$ is non-zero.\label{lem:2_flow_3_flow_seymour}
\end{lemma}

We use the previous lemma to prove an analogous upper bound for the two-dimensional Chebyshev and Manhattan flow number.

\begin{proposition}
    For each bridgeless graph $G$, $\Phi^1_2(G)=\Phi^\infty_2(G)\leq 3$.\label{prop:manhattan_upper_seymour}
\end{proposition}

\begin{proof}
    Let $\varphi_2$ and $\varphi_3$ be a $2$-flow and a $3$-flow from Lemma \ref{lem:2_flow_3_flow_seymour}. Clearly, the flow assigning the vector $(\varphi_2(e), \varphi_3(e))$ to each edge $e \in E(G)$ is a $(3, 2)$-ChNZF. Proposition \ref{prop:manhattan_chebyshev_equivalence} yields then $\Phi^1_2(G)=\Phi^\infty_2(G)\leq 3$.
\end{proof}

It is worth noting that proving a similar result for Manhattan flows without relying on Proposition \ref{prop:manhattan_chebyshev_equivalence} may be less intuitive.

For Euclidean nowhere-zero flows, it has been conjectured that three dimensions are enough to reduce the flow number to its minimal possible value $2$ as mentioned in the Conjecture \ref{conj:3d_flows}. 
In the case of Chebyshev flows, we can establish an analogous result as a direct consequence of Jaeger's 8-flow theorem \cite[pp. 207, 212]{jaeger1979flows} (which guarantees the existence of a nowhere-zero $\mathbb{Z}_2^3$-flow).

\begin{proposition}
	For each bridgeless graph $G$, $\Phi^\infty_3(G) = 2$.
    \label{prop:chebyshev_bound_3_dim}
\end{proposition}

\begin{proof} 
We begin with a $\mathbb{Z}_2^3$-NZF on $G$. Each coordinate of this flow can be interpreted as a $2$-flow, and no edge has all three coordinates equal to zero simultaneously. Therefore, this defines a $(2,3)$-ChNZF on $G$, as desired.
\end{proof}

While we are able to prove the result for the Chebyshev norm, we are not currently able to establish an analogous bound with respect to the Manhattan norm. We therefore state the following as a conjecture.

\begin{conjecture}
For each bridgeless graph $G$, $\Phi_3^1(G) = 2$.
\end{conjecture}

Although we are unable to prove this in full generality, we can establish it as a consequence of the well-known 5-Oriented Cycle Double Cover Conjecture, which we recall here for the reader's convenience.

A usual construction of a flow on a graph consists of assigning flow values to several cycles. For this reason it is convenient to consider the broadest definition of a cycle as a subgraph whose all vertices have even degree. Note that an empty subgraph is also a cycle. An \emph{oriented cycle} is a cycle with an orientation of edges such that each vertex has equal indegree and outdegree. For a positive integer $k$, an \emph{oriented $k$-cycle double-cover} (from now on also referred to as a $k$-OCDC) of $G$ is a multiset of $k$ oriented cycles such that each edge of $G$ occurs in each of its two orientations in exactly one of the cycles. 

It is not known whether every bridgeless graph admits a $k$-OCDC for some value of $k$, but a specific conjecture has been proposed in this direction.

\begin{conjecture} \emph{\cite{archdeacon_5ocdc}}
    Each bridgeless graph has a $5$-OCDC.\label{conj:5ocdc}
\end{conjecture}

We are now in a position to derive an analogue of Proposition \ref{prop:chebyshev_bound_3_dim} for the Manhattan norm.

\begin{proposition}
	If Conjecture \ref{conj:5ocdc} holds, then $\Phi_3^1(G)=2$ for each bridgeless graph $G$.
	\label{prop:manhattan_3d}
\end{proposition}

\begin{proof}
	There exist five points in $\mathbb R^3$ with mutual (Manhattan) distances equal to $1$: for instance, we can consider the five points with coordinates $(\frac12,0,0)$, $(-\frac12,0,0)$, $(0,\frac12,0)$, $(0,-\frac12,0)$, $(0,0,\frac12)$.
    Since $G$ can be covered by five oriented cycles, we can assign these points as flow values to the cycles, defining the flow on each edge as the difference of the flow values assigned to the two cycles that contain it. By construction, each resulting edge flow has Manhattan norm exactly equal to $1$.
\end{proof}

With a stronger assumption of $4$-OCDC, we can provide unit vector flow already in two dimensions.

\begin{proposition}
    \label{prop:4ocdc-2flow}
    If a graph $G$ has a $4$-OCDC, then $\Phi_2^1(G)=\Phi_2^\infty(G)=2$.
\end{proposition}
\begin{proof}
    By assigning flows $\left(\frac12, \frac12\right)$, $\left(\frac12, -\frac12\right)$, $\left(-\frac12, \frac12\right)$, $\left(-\frac12, -\frac12\right)$ to the four oriented cycles in a $4$-OCDC, we get flow values with the Chebyshev norm equal to $1$ on each edge, therefore the graph has a $(2, 2)$-ChNZF.
\end{proof}
In three dimensions, we are able to derive an upper bound without any additional assumptions.

\begin{proposition}
\label{prop:phi_1_3_bound}
For any bridgeless graph $G$ it holds $\Phi_3^1(G)\leq \frac{5}{2}$.
\end{proposition}
\begin{proof}
Define the following vectors in $\mathbb{R}^3$:
\begin{align*}
&q_1=\left(\frac{3}{8},\ \frac{3}{8},\ -\frac{1}{4}\right),\\
&q_2=\left(\frac{3}{8},\ -\frac{1}{4},\ \frac{3}{8}\right),\\
&q_3=\left(-\frac{1}{4},\ \frac{3}{8},\ \frac{3}{8}\right).
\end{align*}
By the 8-flow theorem, the graph $G$ can be covered by three cycles $C_1, C_2, C_3$. We orient each cycle arbitrarily and assign to each oriented cycle $C_i$ the flow vector $q_i$. The total flow is then defined as the sum of these three cycle flows.
By symmetry, it suffices to find norms of $q_1, q_1+q_2, q_1-q_2, q_1+q_2+q_3$ and $q_1+q_2-q_3$ which are 
\begin{align*}
    \|q_1\|_1&=1,\\
    \|q_1+q_2\|_1&=1,\\
    \|q_1-q_2\|_1&=\frac{5}{4},\\
    \|q_1+q_2+q_3\|_1&=\frac{3}{2},\\
    \|q_1+q_2-q_3\|_1&=\frac{3}{2}.
\end{align*}
By construction and the norm evaluations above, for every edge $e \in E$, the resulting flow satisfies
\[
\|\varphi(e)\|_1 \in \left\{1, \frac{5}{4}, \frac{3}{2}\right\}.
\]

\end{proof}

Building on the previous idea, we consider a more general notion: an \emph{$m$-cycle $k$-cover} is a collection of $m$ cycles that together cover each edge exactly $k$ times. This broader perspective leads to the following result.

\begin{proposition}\label{prop:mcyclekcover}
If a graph $G$ has an $m$-cycle $k$-cover, then for any integer $n$ with $0\leq n < k$ we have $\Phi_{m-n}^1(G)\leq 1+\frac{k}{k-n}$. In particular, $\Phi_m^1(G)=2$.
\end{proposition}
\begin{proof}
Let $C_1,\dots,C_m$ be an $m$-cycle $k$-cover of $G$ and let $b_1,\dots,b_{m-n}\in\mathbb{R}^{m-n}$ be the canonical basis vectors. 
We define the flow $\varphi$ by orienting each cycle $C_i$ and sending a flow of value $\frac{1}{k-n}b_i$ along the cycle $C_i$, for $i=1,\dots,m-n$.
Since the cycles form a $k$-cover, for every edge $e$ the vector $\varphi(e)\in \mathbb{R}^{m-n}$ has $i$ entries of $\pm 1$ for some $k-n\leq i\leq k$ and zeroes otherwise. Thus, it holds $1\leq \|\varphi(e)\|_1\leq \frac{k}{k-n}$ by construction.
\end{proof}
Since any bridgeless graph has a $7$-cycle $4$-cover by Bermond, Jackson and Jaeger \cite{BJJ}, we obtain the following:
\begin{corollary}
    For any bridgeless graph $G$ we have $\Phi_7^1(G) = 2$ and $\Phi_6^1(G)\leq \frac73$.
\end{corollary}

The existence of a $5$-CDC in a graph $G$ directly implies, from Proposition \ref{prop:mcyclekcover}, that $\Phi_5^1(G) = 2$. In the next subsection, we will prove, as a specific instance of a more general argument, that the presence of a $5$-CDC is also sufficient to prove that $\Phi_4^1(G)=2$.

Table \ref{table:upper_bounds} summarizes our results for the best known upper bounds in each dimension. For graphs which are only assumed to be bridgeless, the upper bounds in dimensions $4$ and $5$ are obtained by embedding the flows realized in one dimension lower.

\begin{table}[h!]
\centering
\begin{tabular}{ c || c | c | c | c | c | c }
$d$&2&3&4&5&6&7\\
\hline
4-OCDC&2&2&2&2&2&2\\
5-OCDC&3&2&2&2&2&2\\
5-CDC&3&$\frac{5}{2}$&2&2&2&2\\
Bridgeless&3&$\frac{5}{2}$&$\frac{5}{2}$&$\frac{5}{2}$&$\frac{7}{3}$&2\\
\end{tabular}
\caption{Best known upper bounds on $\Phi_d^1(G)$, depending on the assumptions on the graph $G$ (left column).}
\label{table:upper_bounds}
\end{table}

\subsection{Manhattan flows via Hadamard matrices}

A {\it Hadamard matrix} is a square matrix $H$ of order $n$ with entries $h_{ij} \in \{+1, -1\}$, which satisfies the property
\[
HH^T = nI_n
\]
where $H^T$ is the transpose of $H$ and $I_n$ is the identity matrix of order $n$.

This property implies that the rows (and consequently, the columns) of a Hadamard matrix are mutually orthogonal. In other words, if you compare two different rows, the number of positions where their elements are identical is $n/2$, and the number of positions where they differ is also $n/2$.

Hadamard matrices can only exist for orders $n=1$, $n=2$, or for orders that are a multiple of 4 (i.e., $n \equiv 0 \pmod{4}$). 

Below are examples of a Hadamard matrix of order 4 and one of order 8. For clarity, $+$ denotes $1$ and $-$ denotes $-1$.
\[
H_4 = \begin{pmatrix}
+ & + & + & + \\
+ & - & + & - \\
+ & + & - & - \\
+ & - & - & +
\end{pmatrix} 
\]
\[
H_8 = \begin{pmatrix}
H_4 & H_4 \\
H_4 & -H_4
\end{pmatrix}
= \begin{pmatrix}
+ & + & + & + & + & + & + & + \\
+ & - & + & - & + & - & + & - \\
+ & + & - & - & + & + & - & - \\
+ & - & - & + & + & - & - & + \\
+ & + & + & + & - & - & - & - \\
+ & - & + & - & - & + & - & + \\
+ & + & - & - & - & - & + & + \\
+ & - & - & + & - & + & + & -
\end{pmatrix}
\]

\begin{proposition}
Let $G$ be a graph and let $m \geq 2$ be an integer such that:
\begin{itemize}
\item there exists an Hadamard matrix of order $m-1$;
\item the graph $G$ admits an $m$-cycle double cover.
\end{itemize}
Then, $\Phi_{m-1}^1(G)=2$.
\end{proposition}
\begin{proof}
Let $C_1,\dots,C_m$ be an $m$-cycle $2$-cover of $G$ and let $b_1,\dots,b_{m-1}\in\mathbb{R}^{m-1}$ be the vectors corresponding to the rows of an Hadamard matrix of order $m-1$. 
We define the flow $\varphi$ by orienting each cycle $C_i$ and assigning a flow of value $\frac{1}{m-1}b_i$ along cycle $C_i$, for $i=1,\dots,m-1$.
For every edge $e$, we consider two cases according to whether $e$ belongs to $C_m$ or not. If $e$ does not belong to $C_m$, then the vector $\varphi(e)\in \mathbb{R}^{m-1}$ has $\frac{m-1}{2}$ entries of $\pm \frac{2}{m-1}$, with all other entries being zeroes, since any $b_i$, $b_j$ agree in exactly $(m-1)/2$ coordinates.
While, if $e$ belongs to $C_m$, then $\varphi(e)\in \mathbb{R}^{m-1}$ has all entries equal to $\pm \frac{1}{m-1}$.   
Thus, in both cases, it holds $\|\varphi(e)\|_1=1$.
\end{proof}

\section{Two-dimensional case}
\label{sec:2dim}

In this section, we focus on the two-dimensional case. Our primary goal is to establish results for cubic graphs with respect to the Manhattan and Chebyshev norms. We believe that studying these norms can also lead to new bounds in the classical Euclidean setting.

\begin{theorem}
    A cubic graph has a $(2, 2)$-ChNZF if and only if it is $3$-edge-colourable.\label{th:2_mnzf_iff_3_col}
\end{theorem}

\begin{proof}
If a cubic graph is $3$-edge-colorable, then it admits a $4$-oriented cycle double cover ($4$-OCDC), see \cite{zhang_oriented_covers}.
By Proposition~\ref{prop:4ocdc-2flow} it has a $(2,2)$-ChNZF.

For the converse implication, consider a cubic graph with a $(2, 2)$-ChNZF. By Proposition \ref{prop:chebyshev_normal_form}, we can assume that the corresponding map $\varphi\colon E\to\mathbb{R}^2$ takes values in the set of points $x\in\mathbb{Z}^2$ with $\|x\|_{\infty}=1$.
Note that this set consists of exactly eight points, which we partition as follows:
\begin{align*}
&C_1=\{(1,0),(-1,0)\},\\
&C_2=\{(0,1),(0,-1)\},\\
&C_3=\{(1,1),(-1,1), (1,-1),(-1,-1)\}.
\end{align*}
This defines a map $c\colon E\to \{1,2,3\}$ by setting $c(e)=i\ $ if $\varphi(e)\in C_i$. We claim that $c$ is an edge-coloring of $G$. Indeed, if two adjacent edges had flow values in the same set $C_i$, their sum and difference would be one of the vectors $(0,0), (\pm 2,0), (0,\pm 2)$, none of which can be the flow value of the third edge incident to their common vertex.
\end{proof}

As is often the case in the study of flows, cubic graphs that are not 3-edge-colorable represent the challenging instances of this problem. In this context, we define a \emph{snark} as any $2$-connected simple cubic graph that is not 3-edge-colorable. Here, we provide a lower bound for $\Phi_2^{\infty}(G)$ for any snark $G$ in terms of its order.
The proof of Proposition \ref{prop:lowerboundsnark} is a generalization of the argument presented by Tabarelli in \cite{tabarelli_phd}, which was originally applied only to the Petersen graph.

\begin{proposition}\label{prop:lowerboundsnark}
	Let $G$ denote a snark of order $n$. Then, $$\Phi_2^{\infty}(G) \geq 2+ \frac{1}{\left\lfloor\frac{n-2}{4}\right\rfloor}.$$

\end{proposition}

\begin{proof}
Throughout the proof, we set $\xi = \left\lfloor \frac{n - 2}{4} \right\rfloor$.

    Assume by contradiction that there exists a $2$-dimensional flow $\varphi$ of $G$ such that $\varphi = (\varphi_1(e), \varphi_2(e))$ for each edge $e \in E(G)$, with $\|\varphi(e)\|_\infty\geq 1$ and $\varphi_i(e) \in (-1-1/\xi, 1+1/\xi)$ for $i = 1, 2$.

    We say an edge $e \in E(G)$ is \emph{nice} with respect to $\varphi_i$ if $|\varphi_i(e)| \in [1, 1+1/\xi)$, \emph{bad} otherwise. Observe that an edge $e$ can be nice with respect to both $\varphi_1$ and $\varphi_2$, but it cannot be bad with respect to both $\varphi_1$ and $\varphi_2$, for otherwise $\|\varphi(e)\|_\infty<1$.

    Denote by $B_i$ the subgraph of $G$ induced by the bad edges with respect to $\varphi_i$ and by $N_i$ the one induced by the nice edges with respect to $\varphi_i$, $i = 1, 2$. By previous observation at least one of $B_1$ and $B_2$ has at most $\lfloor|E(G)|/2\rfloor=\left\lfloor\frac{3n}4\right\rfloor$ edges, say $B_1$.

    \begin{claim}
        $B_i$ is a spanning subgraph of $G$ and $\Delta(B_i) \leq 2$, for $i = 1, 2$.
    \end{claim}
    \begin{proof}
        Observe that $\Delta(N_i) \leq 2$, because the sum of three real numbers all with absolute value in the interval $[1, 1+1/\xi)$ cannot give $0$ as a result, making the Kirkoff's law impossible to be satisfied by $\varphi$ around a vertex of $G$. Hence $B_i$ is spanning, for otherwise $\Delta(N_i) = 3$ and $\Delta(B_i) \leq 2$, for otherwise $\Delta(N_{3-i}) = 3$.
    \end{proof}

    \begin{claim}
        If $C\subseteq E(N_i)$ is an odd edge-cut of $G$ for $i=1, 2$, then $|C|\geq 2\xi+3$.
    \end{claim}
    \begin{proof}
        Consider an odd edge-cut $C$ of $G$ containing $2k+1$ good edges, separating components $G_1, G_2$. WLOG assume $\varphi_i(e)\geq0$ for every edge $e$ of $C$ and more edges are directed towards $G_2$. Then, there are at least $k+1$ edges towards $G_2$, resulting in total inflow at least $(k+1)\cdot1$. Analogously, the total outflow is strictly less than $k\cdot(1+1/\xi)$. Together with the Kirkoff's law, this leads to $k+1<k\cdot(1+1/\xi)$, which is equivalent to $k>\xi$.
    \end{proof}

    \begin{claim}
        A path of length $2k$ cannot be a connected component of $B_i$, for $k=1,2,\dots, \xi-1$ and $i = 1, 2$.
    \end{claim}
    \begin{proof}
        For the sake of contradiction assume there is a path of a length $2k$, $k<\xi$ in $B_i$. Then the edges adjacent to this path are in $N_i$ and they form an odd edge-cut of $G$, containing at most $2k+3<2\xi+3$ edges, which is in contradiction with the Claim 2.
    \end{proof}

    \begin{claim}
        A cycle of length $2k+1$ cannot be a connected component of $B_i$, for $k=1,2,\dots, \xi$ and $i = 1, 2$.
    \end{claim}
    \begin{proof}
        For the sake of contradiction assume there is a cycle of a length $2k+1$, $k<\xi+1$ in $B_i$. Then the edges adjacent to this cycle are in $N_i$ and they form an odd edge-cut of $\Gamma$, containing at most $2k+1<2\xi+3$ edges, which is in contradiction with the Claim 2.
    \end{proof}

    \begin{claim}
        $E(B_i)$ cannot contain a perfect matching of $G$, for $i = 1, 2$.
    \end{claim}
    \begin{proof}
        For the sake of contradiction assume that $E(B_i)$ contains a perfect matching $M$ of $G$. Then also $E(N_{3-i})$ contains a perfect matching $M$ of $G$. Next, $F=E(G)\setminus M$ is a $2$-factor of $G$. Note that $F$ must contain an odd cycle of length $2k+1\leq\frac n2$. This is equivalent to $k\leq\xi$. Then the edges adjacent to this cycle are in $N_{3-i}$ and they form an odd edge-cut of $G$, containing at most $2k+1\leq 2\xi+1$ edges, which is in contradiction with the Claim 2.
    \end{proof}

    Note that by the Claim 3, each even path in $B_1$ contains at least $2\xi$ edges. Similarly by the Claim 4, each odd cycle in $B_1$ contains at least $2\xi+3$ edges. Let \emph{odd components} denote odd cycles and even paths. Since $4\xi=4\lfloor\frac{n-2}{4}\rfloor>\lfloor\frac{3n}{4}\rfloor\geq E(B_1)$ obviously holds for any $n\geq 10$, $B_1$ may contain at most one odd component. On the other hand, $B_1$ contains an even number of odd components. As a result, $B_1$ contains only odd paths and even cycles, but then also a perfect matching of $G$, which is in contradiction with the Claim 5.
\end{proof}

Now, we summarize some computational results about the value of  $\Phi_2^{\infty}$ for snarks of small order. A comprehensive database of all such graphs can be found in \cite{Hog}. Cyclically $4$-edge-connected snarks are available directly. We have obtained snarks with smaller cyclic edge-connectivity by filtering cubic graphs.

Table~\ref{table:snarks_flows} presents the Chebyshev (and thus also Manhattan) flow numbers of all cyclically $4$-edge-connected snarks of order at most $26$.
Notably, $7$ snarks reach the lower bound from Proposition \ref{prop:lowerboundsnark}: the Petersen graph, the Blanu\v sa snarks of order $18$ and four snarks of order $20$ (refer to the bold entries in the table).
We are not aware of any other snark, for which the bound is tight, which we further address in Problem \ref{prob:lower_bound_tight}.

\begin{table}[h!]
\centering
\begin{tabular}{|c|c|c|c|c|}
\hline
\textbf{Order$\backslash \Phi_2^{\infty}$} & $2+1/4$ & $2+1/3$ & $2+1/2$ & \textbf{Total} \\
\hline
10 & - & - & \textbf{1} & 1 \\
\hline
18 & \textbf{2} & - & - & 2 \\
\hline
20 & \textbf{4} & 2 & - & 6 \\
\hline
22 & 23 & 8 & - & 31\\
\hline
24 & 135 & 20 & - & 155 \\
\hline
26 & 1181 & 116 & - & 1297 \\
\hline
\end{tabular}
\caption{Two-dimensional Chebyshev flow numbers of small cyclically $4$-edge-connected snarks.}
\label{table:snarks_flows}
\end{table}

\begin{figure}[ht]
\centering
\includegraphics[width=8cm]{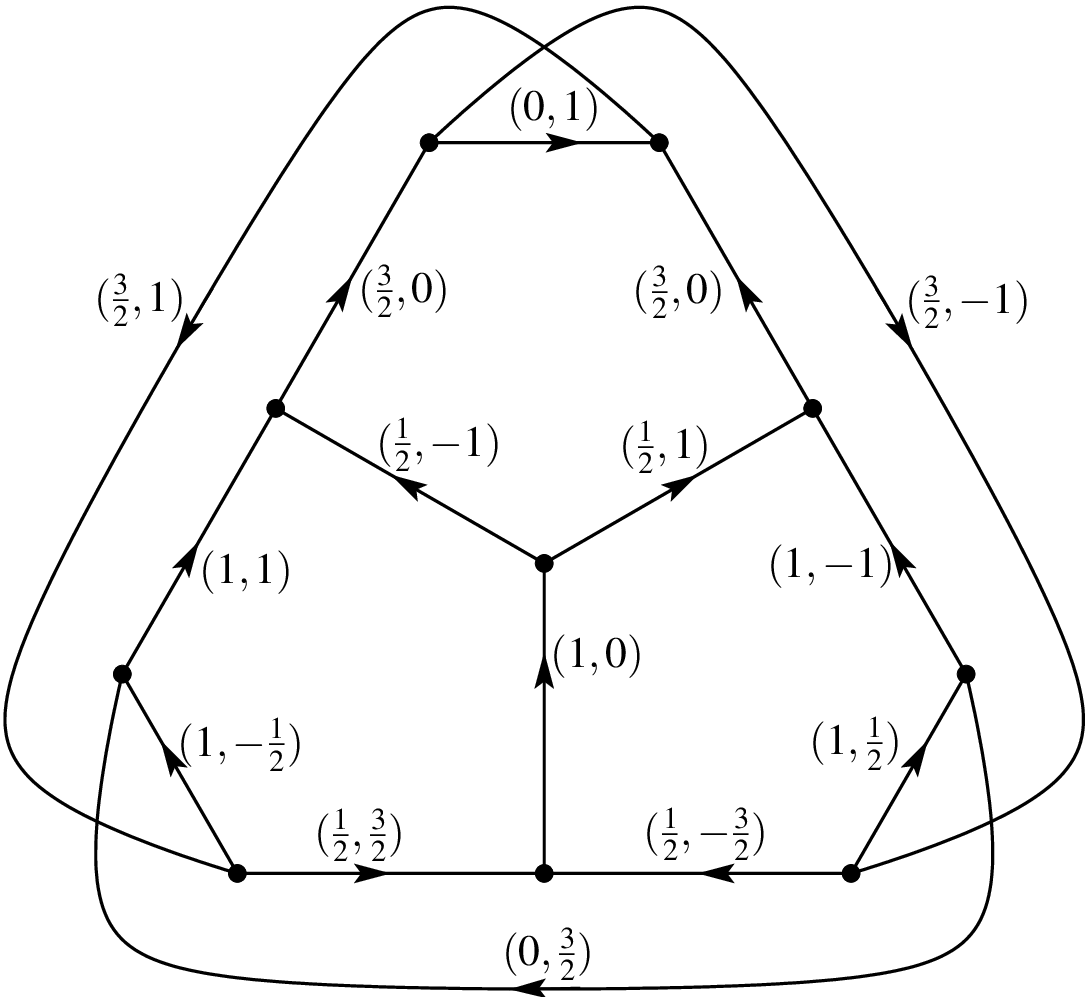}
\caption{A $(\frac{5}{2},2)$-ChNZF of the Petersen graph.}\label{fig:petersen}
\end{figure}

We have also tested all snarks up to order $20$ that are not cyclically $4$-edge-connected (usually called \emph{trivial} snarks). All of them have $\Phi_2^\infty$ equal to $5/2$ except for $9$ graphs of order $20$ with $\Phi_2^\infty$ equal to $9/4$ that are obtained by expanding a vertex to a triangle in one of the Blanu\v sa snarks, hence having cyclic connectivity equal to $3$.

Computational results suggest that $\Phi_2^{\infty}(G)$ could be at most $\frac{5}{2}$ and in particular in the form $2+\frac1k$, $k \in \mathbb Z_{>0}$, for all snarks. We leave it as an interesting open problem, see Problem \ref{prob:manhattan_greater_2.5}.

These results have been obtained by translating the problem to a mixed integer linear programming instance.
Let $G=(V,E)$ be a graph, where $V=\{1,\dots,v\}$. We orient the edges from the smaller to the larger vertex index, e.g.\ from $1$ to $3$. 
The flow value on an edge $e\in E$ is given by $\left(x^+_e-x^-_e,y^+_e-y^-_e\right)$, where $x_e^+,x_e^-,y_e^+,y_e^-\in\mathbb{R}_{\ge 0}$. 
In order to ensure $x_e^+\cdot x_e^-=0$ we require  $x_e^+-\Lambda u_e^x \le 0$ and $x_e^-+\Lambda u_e^x \le \Lambda$ using a binary variable $u_e^x$ and an upper bound $\Lambda$ on the particular coordinates. By Proposition~\ref{prop:manhattan_upper_seymour}, we use $\Lambda = 2$. Similarly, we require $y_e^+-\Lambda u_e^y \le 0$ and $y_e^-+\Lambda u_e^y \le \Lambda$ using a binary variable $u_e^y$.

We minimize $z$ subject to $x_e^++x_e^--z\le 0$ and $y_e^++y_e^--z\le 0$ for all $e\in E$. In order to ensure $\left\Vert\left(x^+_e-x^-_e,y^+_e-y^-_e\right)\right\Vert_\infty\ge 1$ 
we utilize $v_e^1,v_e^2,v_e^3,v_e^4\in\{0,1\}$ and require $x_e^++v_e^1\ge 1$, $x_e^-+v_e^2\ge 1$, $y_e^++v_e^3\ge 1$, $y_e^-+v_e^4\ge 1$, and $v_e^1+v_e^2+v_e^3+v_e^4\le 3$ 
for all edges $e\in E$. Finally, flow conservation for vertex $i\in V$ is modeled via
\begin{equation*}
  \sum_{j\in V\,:\, i<j,\{i,j\}\in E} x_{\{i,j\}}^+  \,+\, \sum_{j\in V\,:\, i>j,\{i,j\}\in E} x_{\{i,j\}}^- \,=\, 0 
\end{equation*} 
and similarly for the $y$ variables.

Since we have obtained a bounded mixed-integer linear program with integer coefficients, we can argue that it's optimal solution -- the number $\Phi_2^\infty(G)$ -- is rational.
This can be easily generalised to an alternate proof of rationality of $\Phi_d^\infty$ for any dimension $d$. This approach is more suitable for effective computer experiments compared to Theorem~\ref{th:rationality}.

\section{{\it t}-flow-pairs}
\label{sec:circular_decompositions}

A common approach in constructing flows is to cover the edges of a graph by two or more flows such as Seymour's Lemma~\ref{lem:2_flow_3_flow_seymour}, which provides a covering with a $2$-flow and a $3$-flow. Since this approach proved to be useful, it was generalised by Xie and Zhang \cite{Xie} who restricted possible pairs of flow values.
The study of two-di\-men\-sio\-nal flows naturally relates to the study of flow pairs. Thus, in this section, we introduce another generalisation of the flow pair from Seymour's lemma and propose a conjecture that implies a stronger upper bound on the two-dimensional Chebyshev flow number and also the $5$-flow conjecture.

\begin{definition}
Let $0 < t \leq 1$ be a rational number. A \emph{$t$\nobreakdash-flow\nobreakdash-pair} of a graph is a pair consisting of a $2$\nobreakdash-flow~$\varphi_2$ and a $(p+q+1)$\nobreakdash-flow~$\varphi_{p+q+1}$ with the same orientation for some positive integers $p$ and $q$ such that $t = p/q$, with the property that for every edge~$e$, if $\varphi_2(e) = 0$, then $|\varphi_{p+q+1}(e)| \geq q$.
\end{definition}

Note that the values of $p$ and $q$ in the previous definition are not uniquely determined; however, what matters for our purposes is their ratio $t$ as we show in the following paragraph.

If we have a $(dp+dq+1)$-flow $\varphi$, by a similar argument as in Proposition \ref{prop:chebyshev_normal_form}, either all flow values are multiples of $d$ or we can find a cycle, where no flow value is divisible by $d$. Then, we evaluate differences of those flow values and their nearest multiples of $d$. By picking the smallest one and sending it along the cycle, no flow value exceeds the bounds $dq$ and $dp+dq$ (since both are multiples of $d$) and the number of flow values not divisible by $d$ decreases. Hence we are able to adjust the flow $\varphi$ to a flow $\varphi'$ such that all flow values are divisible by $d$, and so we can construct a $(p+q+1)$-flow $\varphi''$ by dividing each flow value by $d$. Moreover, it is easy to see, that $|\varphi''|\geq q$ whenever $|\varphi|\geq dq$, and thus also the relation between the $2$-flow and $\varphi''$ is fulfilled. The converse direction is trivial.

\begin{proposition}
    Consider a bridgeless graph $G$ with a $t$\nobreakdash-flow\nobreakdash-pair. Then, $G$ admits a $\left(2+t,2\right)$-ChNZF and a $\left(4+2t, 1\right)$-NZF.\label{prop:nzf_from_two_circulations_generalised}
\end{proposition}

\begin{proof}
Let $t = p/q$ for some fixed positive integers $p$ and $q$, and consider an arbitrary $t$-flow-pair $(\varphi_2, \varphi_{p+q+1})$ as in the definition. 
    Let $\varphi\colon E\to\mathbb R^2$ be the function with
    $\varphi(e)=\left(\varphi_2(e), \varphi_{p+q+1}(e)/q\right)$ for each edge $e$ of $G$. It is easy to see that $\varphi$ is a flow on $G$ and that for each edge $e$, $1\leq\|\varphi(e)\|_\infty\leq1+p/q$ holds true. Therefore, $\varphi$ is a $\left(2+t,2\right)$\nobreakdash-ChNZF on $G$.

    Now consider a flow $\varphi'$ with$$\varphi'(e)=\left(2+p/q\right)\varphi_2(e)+\varphi_{p+q+1}(e)/q.$$ Trivially, $\varphi'$ is a flow on $G$. Next,
    \begin{align*}
        |\varphi'(e)|&=\left|\left(2+p/q\right)\varphi_2(e)+\varphi_{p+q+1}(e)/q\right|\leq\\
        &\leq\left(2+p/q\right)\cdot|\varphi_2(e)|+|\varphi_{p+q+1}(e)|/q\leq\\
        &\leq\left(2+p/q\right)\cdot1+(p+q)/q=3+2p/q
    \end{align*}
    for each edge $e$. When $\varphi_2(e)$ is zero, then $|\varphi'(e)|=|\varphi_{p+q+1}(e)|/q\geq1$. On the other hand, assuming $|\varphi_2(e)|=1$, we get
    \begin{align*}
        |\varphi'(e)|&=\left|\left(2+p/q\right)\varphi_2(e)-\left(-\varphi_{p+q+1}(e)/q\right)\right|\geq\\
        &\geq\left(2+p/q\right)\cdot|\varphi_2(e)|-\left|-\varphi_{p+q+1}(e)\right|/q\geq\\
        &\geq\left(2+p/q\right)\cdot1-(p+q)/q=1
    \end{align*}
    and hence $\varphi'$ is a $\left(4+2t,1\right)$-NZF on $G$.
\end{proof}

    We now point out some notable special cases of the previous proposition. The case $t=1$, in particular with $(p,q) = (1,1)$, recovers both Seymour's 6-flow theorem and Proposition \ref{prop:manhattan_upper_seymour} on $(3,2)$-Chebyshev flows, directly from Lemma \ref{lem:2_flow_3_flow_seymour}. 
    
The case $t=1/2$ corresponds to Tutte's 5-flow conjecture and the $(5/2, 2)$-Chebyshev flow conjecture (Conjecture \ref{prob:manhattan_greater_2.5}). Consequently, proving the existence of a $1/2$-flow-pair for every bridgeless graph would be sufficient to establish both results.

Using an algorithm that exhaustively searches for flow-pairs by solving instances of the satisfiability (SAT) problem, we have found a $1/2$\nobreakdash-flow-pair for all cyclically $4$-edge-connected snarks up to 34 vertices, for all cyclically $4$-edge-connected snarks with $\Phi_1(G) = 5$ and girth at least $5$ on 36 vertices, and for snarks of oddness 4 and cyclic connectivity $4$ on 44 vertices.

We therefore leave the existence of a $1/2$-flow-pair in every bridgeless graph as a potential strengthening of Tutte's 5-flow conjecture (see Conjecture~\ref{conj:12-flow-pair}).

Now we show an additional property of flow-pairs on snarks, which may offer further insight into proving their existence.
\begin{lemma}\label{lem:2factor}
	Let $0 < t < 1$ be a positive rational number. Consider a snark $G$ with a $t$-flow-pair. Then the support of $\varphi_2$ is a $2$-factor of $G$.\label{lem:2_circ_snark_2_factor}
\end{lemma}

\begin{proof}
Let $t = p/q$ for some fixed positive integers $p<q$, and consider an arbitrary $t$-flow-pair $(\varphi_2, \varphi_{p+q+1})$ as in the definition. 
	For the sake of contradiction, consider a vertex $v$ in which the number of non-zero $\varphi_2$ edges differs from two. Let $e_1$, $e_2$, $e_3$ denote the edges incident with $v$. Since $v$ fulfills the flow conservation constraint, the number of edges $e$ incident with $v$ fulfilling $\varphi_2(e)\neq 0$ is obviously even. Therefore, it has to be zero and hence, $\varphi_2(e_i)=0$ for $i=1,2,3$. This necessarily leads to $q\leq|\varphi_{p+q+1}(e_i)|\leq p+q$ for $i=1,2,3$. Without loss of generality assume that in the orientation of $\varphi_{p+q+1}$, these $3$ edges are all outgoing from $v$. Then from the flow conservation constraint, $\varphi_{p+q+1}(e_1)+\varphi_{p+q+1}(e_2)+\varphi_{p+q+1}(e_3)=0$. This may be possible only if the three flow values do not have the same sign. Without loss of generality assume that only $\varphi_{p+q+1}(e_3)$ is negative. But then,
\begin{equation*}
0=\varphi_{p+q+1}(e_1)+\varphi_{p+q+1}(e_2)+\varphi_{p+q+1}(e_3)\geq q+q+(-p-q)=q-p,
\end{equation*}
which is a sought contradiction.
\end{proof}

A direct consequence of Lemma~\ref{lem:2factor} is the following equivalent condition for the existence of a $t$-flow-pair in snarks, valid for all $t < 1$.

\begin{corollary}
Let $t = p/q < 1$ be a positive rational number. A snark $G$ admits a $t$-flow-pair if and only if there exist a perfect matching $M$ of $G$ and a $(p+q+1)$-flow~$\varphi_{p+q+1}$ such that $\left|\varphi_{p+q+1}(e)\right| \geq q$ for every edge $e \in M$.
\end{corollary}

We conclude this section by comparing our new notion of $t$-flow-pairs to parity-pair-covers introduced by Xie and Zhang \cite{Xie}.
A \emph{$(h, k)$-flow parity-pair-cover} of a graph $G$ is a pair of an $h$-flow $\varphi_1$ and a $k$-flow $\varphi_2$ on $G$ such that for each edge $e \in E(G)$
the values $\varphi_1(e)$ and $\varphi_2(e)$ have the same parity and at least one of them is nonzero.
They proved that each graph with a $(3, 3)$-flow parity-pair-cover admits a nowhere-zero $5$-flow and also a $5$-cycle double cover. Also, they proposed a conjecture that each bridgeless graph admits a $(3, 3)$-flow parity-pair-cover, which is an analogy of our Conjecture \ref{conj:12-flow-pair}.

Both these conjectures offer a possible approach to proving the $5$-flow conjecture. On the one hand, no $(5/2, 2)$-ChNZF can be obtained by a linear combination of two flows from some $(3, 3)$-flow parity-pair-cover. On the other hand, we are aware of no simple reason why the existence of a $1/2$-flow-pair on a graph should imply the existence of a $5$-cycle double-cover.

\section{Conjectures and open problems}\label{sec:openproblems}

This paper represents the first investigation of flow problems using norms different from the usual Euclidean one. Beyond the problems already discussed in the preceding sections, we present the following two as the most relevant conjectures derived from our work.

\begin{conjecture}
For all bridgeless graphs $G$, the inequality $\Phi_2^\infty(G)\leq\frac52$ holds. \label{prob:manhattan_greater_2.5}
\end{conjecture}

\begin{conjecture}\label{conj:12-flow-pair}
    For each bridgeless graph there exists a $1/2$-flow-pair.
\end{conjecture}

The former concerns a general upper bound for the $2$-dimensional Chebyshev flow number. The latter, as explained in the preceding section, proves to be a stronger version of both the Tutte's $5$-flow conjecture and Conjecture~\ref{prob:manhattan_greater_2.5}.

As an intermediate result, one could attempt to solve the following problem, which would correspond to a finding between Tutte's 5-flow conjecture and Seymour's 6-flow theorem. 

\begin{problem}\label{problem:weakerthanTutte}
Determine the existence of a $t$-flow-pair for some rational value $t$ with $1/2 < t < 1$.
\end{problem}

Note that the $(2+t,2)$\nobreakdash-ChNZF and $(4+2t, 1)$-NZF constructed from the $t$-flow pair in Proposition~\ref{prop:nzf_from_two_circulations_generalised} are in an interesting relation $4+2t=2(2+t)$. Hence, we have examined whether any of inequalities between $\Phi_1$ and $2\Phi_2^\infty$ hold generally.

For all snarks with $\Phi_1=5$ and the number of vertices between $28$ and $36$, we have found a $(7/3, 2)$-ChNZF and thus, $\Phi_1\leq2\Phi_2^\infty$ does not hold in general.
For the converse inequality $\Phi_1\geq2\Phi_2^\infty$, all snarks on $20$ vertices satisfy $\Phi_1=9/2$ \cite{cfn_computational_results}, but from our computational results, four of them have $\Phi_2^\infty=7/3$. 

Therefore, none of the considered inequalities holds. This part also inspired us to propose the Problem \ref{prob:manhattan_half_circular}. 

\begin{problem}
    Determine all possible values of the ratio $\Phi_1(G)/\Phi_2^\infty(G)$, where $G$ is a bridgeless graph / snark.\label{prob:manhattan_half_circular}

\end{problem}

The study of multidimensional flows with respect to norms other than the Euclidean norm has revealed interesting structural properties and notable connections to classical problems in graph theory. These findings suggest that such investigations are intrinsically relevant and merit further exploration. In this context, we consider the family of $p$-norms, defined for a vector $\boldsymbol{x} = (x_1, \dots, x_d) \in \mathbb{R}^d$ by
\[
\|\boldsymbol{x}\|_p = \left( \sum_{i=1}^d |x_i|^p \right)^{1/p},
\]
with the cases $p = 1$, $p = 2$, and $p = \infty$ corresponding to the Manhattan, Euclidean, and Chebyshev norms, respectively. 
We conjecture that the flow value exhibits a unimodular behavior with respect to $p$, and then the values obtained under the Chebyshev norm and the Manhattan norm provide lower bounds for the classical Euclidean case, which remains largely open and poorly understood.
More precisely, we propose the following conjecture.

\begin{conjecture}
For any fixed dimension $d \geq 1$ and any fixed graph $G$, the value of the multidimensional flow number $\Phi_d^p(G)$, viewed as a function of the norm parameter $p \in [1, \infty]$, is unimodal. That is, there exists $p^* \in [1, \infty]$ such that $\Phi_d^p(G)$ is non-decreasing for $p \leq p^*$ and non-increasing for $p \geq p^*$. Moreover, $p^*=2$ holds.
\end{conjecture}

If this conjecture is verified, it would be interesting to further investigate lower bounds in the $2$-dimensional case, which leads us to pose the following problem:

\begin{problem} \label{prob:lower_bound_tight}
Are there infinitely many values of $n$, for which there exists a snark of order $n$ that attains the bound in Proposition~\ref{prop:lowerboundsnark}?
\end{problem}

\section{Acknowledgements}

The fourth author is funded by the EU NextGenerationEU through the Recovery and Resilience Plan for Slovakia under the project No. 09I03-03-V05-00012.

 \bibliographystyle{plain}
 \bibliography{bibliography}

\end{document}

%% file: images/manhattan_chebyshev_equivalence.tex
\centering\begin{tikzpicture}[auto]

    \filldraw[fill=lightgray]
    (-4.5, 0) -- (-3, 1.5) -- (-1.5, 0) -- (-3, -1.5) -- cycle;
    \filldraw[fill=white]
    (-4.125, 0) -- (-3, 1.125) -- (-1.875, 0) -- (-3, -1.125) -- cycle;

    \filldraw[fill=lightgray]
    (0.5, 1.5) -- (3.5, 1.5) -- (3.5, -1.5) -- (0.5, -1.5) -- cycle;
    \filldraw[fill=white]
    (0.875, 1.125) -- (3.125, 1.125) -- (3.125, -1.125) -- (0.875, -1.125) -- cycle;


    \path[latex-latex, thick]
    (-3, 0) edge node [below right]{$r-1$} (-4.5, 0)
    (-3, 0) edge node[right] {$1$} (-3, 1.125)
    (2, 0) edge node {$1$} (0.875, 0)
    (2, 0) edge node[right] {$r-1$} (2, 1.5);
    
\end{tikzpicture}